\documentclass[a4paper, 12pt]{elsarticle}
\journal{Topology and its Applications}

\usepackage{amsfonts,amsmath,amssymb,amsthm,latexsym}
\usepackage{latexsym}
\usepackage[T2A]{fontenc}
\usepackage[cp1251]{inputenc}
\usepackage[english]{babel}

\newtheorem{definition}{Definition}[section]
\newtheorem{theorem}[definition]{Theorem}
\newtheorem{lemma}[definition]{Lemma}
\newtheorem{proposition}[definition]{Proposition}
\newtheorem{corollary}[definition]{Corollary}

\theoremstyle{definition}

\numberwithin{equation}{section}

\makeatletter

\def\ps@pprintTitle{%
     \let\@oddhead\@empty
     \let\@evenhead\@empty
     \def\@oddfoot{
       \centerline{\thepage}}%
     \let\@evenfoot\@oddfoot}

\makeatother

\begin{document}
\begin{frontmatter}
\title{FIRST BAIRE CLASS FUNCTIONS IN THE PLURI-FINE TOPOLOGY}

\author[OD]{Oleksiy Dovgoshey\corref{cor1}}
\ead{aleksdov@mail.ru}

\author[MK]{Mehmet K\"{u}\c{c}\"{u}kaslan}
\ead{mkucukaslan@mersin.edu.tr}

\author[JR]{Juhani Riihentaus}
\ead{juhani.riihentaus@gmail.com}

\cortext[cor1]{Corresponding author}

\address[OD]{The Division of Applied Problems in Contemporary Analysis, \\ Institute of Mathematics of NASU, \\Tereschenkivska str. 3, Kyiv-4, 01601 Ukraine}
\address[MK]{Department of Mathematics, Faculty of Art and Sciences, \\Mersin University, Mersin 33342, Turkey}
\address[JR]{Department of Mathematical Sciences, University of Oulu, \\P. O. Box 3000, FI - 90014; \\ Department of Physics and mathematics, University of Eastern Finland, \\P. O. Box 111, FI - 80101, Joensuu, Finland}

\begin{abstract}
Let $B_{1}(\Omega, \mathbb R)$ be the first Baire class of real functions in the pluri-fine topology on an open set  $\Omega \subseteq \mathbb C^{n}$ and let $H_{1}^{*}(\Omega, \mathbb R)$ be the first functional Lebesgue class of real functions in the same topology. We prove the equality $B_{1}(\Omega, \mathbb R)=H_{1}^{*}(\Omega, \mathbb R)$  and show that for every $f\in B_{1}(\Omega, \mathbb R)$ there is a separately continuous function $g: \Omega^{2} \to\mathbb R$ in the pluri-fine topology on $\Omega^2$ such that $f$ is the diagonal of $g.$
\end{abstract}

\begin{keyword}
plurisubharmonic function \sep first Baire class function \sep separately continuous function \sep pluri-fine topology \sep first functional Lebesgue class function.
\MSC[2010] Primary 31C10 \sep Secondary 31C40, 54E52.
\end{keyword}

\end{frontmatter}

\section{Introduction }
The first Baire class functions is a classical object for the studies in the Real Analysis, General Topology and Descriptive Set Theory. There exist many interesting characterizations of these functions. Let us denote by $I$ the closed interval $[0, 1]$.

\begin{theorem}\label{t1.1}
The following conditions are equivalent for every $f : I \rightarrow I$.
\begin{enumerate}
\item The function $f$ is a Baire one function.
\item  There is a separately continuous function $g : I \times I \rightarrow I$ such that $f$ is the diagonal of $g$.
\item  Each nonvoid closed set $F \subseteq I$ contains a point $x$ such that the restriction $f|_{F}$ is continuous at $x$.
\item  The sets $f^{-1}(a, 1]$ and $f^{-1}[0, a)$ are $F_{\sigma}$ for every $a\in I$.
\item  For all $a$, $b \in I$ with $a < b$ and for every subset $F \subseteq I$, the sets $f^{-1}[0, a]$ and $f^{-1}[b, 1]$ cannot be simultaneously dense in $F$.
\end{enumerate}
\end{theorem}

It is a classical result in the real functions theory that the diagonals of separately continuous functions of n variables are exactly the $(n-1)$ Baire class functions. See R.~Baire \cite{Ba} for the original proof in the case where $n = 2$,  and H.~Lebesgue \cite{Le1, Le2} and H.~Hahn \cite{Ha} for arbitrary $n\geqslant 2$.  A proof of the equivalence of (i), (iii), (iv) and (v) in the situation of a metrizable strong Baire space can be found, for example, in \cite[Theorem 2.12, p. 55]{LMZ}. The goal of our paper is to find similar characterizations of the first Baire class functions on the topological space $(\Omega, \tau),$ where $\Omega$ is an open subset of $\mathbb C^n$ and $\tau$ is the pluri-fine topology on $\Omega$. The pluri-fine topology $\tau$ is the coarsets topology on $\Omega$ such that all plurisubharmonic functions on $\Omega$ are continuous. The topology $\tau$ was introduced by B.~Fuglede in \cite{Fug} as a basis for a fine analytic structure in $\mathbb C^{n}$.  E.~Bedfor and B.~Taylor note in~\cite{BT} that the pluri-fine topology is Baire and has the quasi-Lindel\"{o}f property. S.~El.~Marzguioui and J.~Wiegerinck proved in \cite{MW1} that $\tau$ is locally connected and, consequently, the connected components of open sets are open in $\tau$ (see also \cite{MW2}). It should be noted that $\tau$ is not metrizable (see Corollary~\ref{c1.7} below). Thus, it is not clear whether the formulated above characterizations of the first Baire class functions are valid for $(\Omega, \tau).$

Let us recall some definitions.

Let $X$ be an arbitrary nonvoid set. For integer $m\geqslant 2$ the set $\Delta_{m}$ of all $m$-tuples $(x, ..., x), \, x\in X$, is by definition, the \emph{diagonal} of $X^m$. The mapping $d_m : X \rightarrow X^{m},\, d_{m}(x) = (x, ..., x)$, is called the \emph{diagonal mapping} and, for every function $f : X^{m} \rightarrow Y$, the composition $f \circ d_m$,
\begin{equation*}
X\ni x\longmapsto f(x, ..., x)\in Y
\end{equation*}
is, by definition, the \emph{diagonal} of $f$.

Let $X$ and $Y$ be topological spaces. A function $f : X \rightarrow Y$ is a \emph{first Baire class function} if there exists a sequence $(f_n)_{n\in\mathbb N}$ of continuous functions $f_n : X \rightarrow Y$ such that the limit relation
\begin{equation}\label{eq1.1}
f(x)=\lim_{n\to\infty}f_{n}(x)
\end{equation}
holds for every $x\in X$. Similarly, for an integer number $m\geqslant 2$, a function $f : X \rightarrow Y$ belongs to the \emph{$m$-Baire class functions}, if \eqref{eq1.1} holds with a sequence $(f_n)_{n\in\mathbb N}$ such that each of $f_n$ is in a Baire class less than $m$. A function $f : X \rightarrow Y$ is a \emph{first functional Lebesgue class function}, if for every open subset $G$ of the space $Y$, the inverse image $f^{-1}(G)$ is a countable union of functionally closed subsets of
$X.$ We will denote by $B_{1}(X, Y )$ (by $H_{1}^{*}(X, Y)$) the set of first Baire (first functional Lebesgue) class functions from $X$ to $Y$ and by $F_{\sigma}^{*}$ ($G_{\delta}^{*}$) the set of all countable unions (countable intersections) of functionally closed (functionally open) subsets of $X$. Thus
\begin{equation}\label{eq1.2}
(f\in H^{*}_{1}(X, Y))\Leftrightarrow (f^{-1}(G)\in F_{\sigma}^{*}\,\,\text{for all open}\,\,G\subseteq Y)
\end{equation}
\begin{equation*}
\Leftrightarrow (f^{-1}(F)\in G_{\delta}^{*}\,\,\text{for all closed}\,\,F\subseteq Y).
\end{equation*}
Recall that a subset $A$ of a topological space $X$ is \emph{functionally closed}, if there is a continuous function $f : X \rightarrow I$ such that $A = f^{-1}(0)$.

\begin{definition}\label{d1.2}
Let $\mu$ be a topology on the Cartesian product $X=\mathop{\prod}\limits_{i=1}^{m}X_{i}$ of nonvoid sets $X_1, ..., X_m,$ $m\geqslant 2$, and let $Y$ be a topological space. A function $f: X\to Y$ is called separately continuous if, for each $m$-tuple $(x_1, ..., x_m)\in X$, the restriction of $f$ to any of the sets
\begin{multline*}
\{(x, x_2, ..., x_m): x\in X_1\},\{(x_1, x, ..., x_m): x\in X_2\}, ...,\\
\{(x_1, ..., x_{m-1}, x): x\in X_m\}
\end{multline*}
is continuous in the subspace topology generated by $\mu$.
\end{definition}

If $a=(a_1, a_2, .., a_n)\in\mathbb C^{n}$, $b=(b_1, b_2, ..., b_n)\in\mathbb C^{n}$ and $z\in\mathbb C$, then we shall write $a+bz$ the for $n$-tuple $(a_1+b_{1}z,a_2+b_{2}z, ..., a_n+b_{n}z)$.

\begin{definition}\label{d1.3}
Let $\mathbb C^{n}$, $\mathbb C$ and $[-\infty, \infty)$ have the Euclidean topologies, $n\geqslant 1$ and let $\Omega\subseteq \mathbb C^n$ be a non-void open set. A function $f: \Omega\to [-\infty, \infty)$ is plurisubharmonic (psh) if $f$ is upper semicontinuous and, for all $a$, $b \in \mathbb C^n$, the function
\begin{equation*}
\mathbb C\ni z\longmapsto f(a+bz)\in[-\infty, \infty)
\end{equation*}
is subharmonic or identically $-\infty$ on every component of the set
$$
\{z\in \mathbb C\colon a+bz\in \Omega\}.
$$
\end{definition}

In what follows, $\tau$ denotes the pluri-fine topology on $\Omega$, i.e., the coarsest topology in which all psh functions are continuous.

As was noted in \cite{MW1}, many results related to the classical fine topology which was introduced by H.~Cartan are valid for the pluri-fine topology. For example, $\tau$ is Hausdorff and completely regular. It is well known that Cartan's fine topology is not metrizable and all compact sets are finite in this topology. The topology $\tau$ also has these properties.

Let $\pi_j: \Omega^{m}\to \Omega$ be the $j$-th projection of $\Omega^{m}$ on $\Omega$, $j\in \{1,...,m\}$.
We identify $\Omega^{m}$ with the corresponding subset of $\mathbb C^{mn}$ and denote by $\tau_{m}$ the pluri-fine topology on $\Omega^{m}$.

\begin{lemma}\label{l2.4}
All projections $\pi_{j}: (\Omega^{m}, \tau_{m})\to (\Omega, \tau)$, $j=1, ..., m$ are continuous.
\end{lemma}
\begin{proof} Let $Y$ be a topological space.
It follows form a general result on the continuity of the mappings to a topological space $X$ with a topology generated by a family $\mathfrak{F}$ of functions $f$ on $X$ (see \cite[p.~31]{En}), that $\psi\colon Y\to X$ is continuous if and only if the composition $f \circ\psi$ is continuous for every $f\in\mathfrak{F}$. Hence, we need show that the functions
\begin{equation}\label{eq2.3}
\Omega^{m} \stackrel{\pi_j }{\longrightarrow} \Omega\stackrel{ f }{\longrightarrow} [-\infty, \infty)
\end{equation}
are continuous in the topology $\tau_{m}$ for every psh function $f$. Note that all projections $\pi_j$ are analytic. Consequently, in \eqref{eq2.3} we have a composition of analytic function with psh function. Since such compositions are psh (see, for example, \cite[p.~228]{Ho}), they are continuous by the definition of pluri-fine topology.
\end{proof}
Substituting $\mathbb C$ instead of $\Omega$ and $n$ instead of $m$ we obtain the following
\begin{corollary}\label{l1.4}
All projections $\pi_{j}: \mathbb C^{n}\rightarrow\mathbb C$, $ j=1, ..., n$, are conti\-nuous mappings with respect to the pluri-fine topologies on $\mathbb C^n$ and $\mathbb C$.
\end{corollary}

\begin{proposition}\label{p1.5}
Let $\Omega$ be a non-void open subset of $\mathbb C^n$ and let $A$ be a compact set in $(\Omega, \tau)$. Then $A$ is finite, $|A|<\infty.$
\end{proposition}
\begin{proof}
If $f$ is a psh function on $\mathbb C^n$, then the restriction $f|_{\Omega}$ is psh on $\Omega$. Hence it is sufficient to show that $|A|<\infty$ for the case $\Omega = \mathbb C^n$. By Corrolary \ref{l1.4} every projection $\pi_{j}$ is continuous. Hence the sets $A_{j}=\pi_{j}(A)$, $j=1, ..., n$, are compact. As was mentioned above every compact set in $(\mathbb C, \tau)$ is finite. Consequently, we have $|A_{j}|<\infty$, $j=1, \ldots, n$. These inequalities and $|A|\leqslant\mathop{\prod}\limits_{j=1}^{n}|A_j|$ imply that $A$ is finite.
\end{proof}

\begin{proposition}\label{p1.6}
Let $\Omega$ be a non-void open subset of $\mathbb C^n$, $n\geqslant 1$. The pluri-fine topology $\tau$ is not first-countable for any $n$.
\end{proposition}
\begin{proof}
Suppose, contrary to our claim, that $\tau$ is first-countable. The topology $\tau$ is Hausdorff.  Since $(\Omega, \tau)$ is not discrete, $\Omega$ contains an accumulation point $a$ which is the limit of a non-constant sequence $(a_k)_{k\in\mathbb N}$ of points of $\Omega$. It is clear that the set
$$
A=\{a\}\cup\left(\bigcup_{k=1}^{\infty}\{a_k\}\right)
$$
is an infinite compact subset of $\Omega$. The last statement contradicts Proposition 1.6.
\end{proof}

\begin{corollary}\label{c1.7}
The pluri-fine topology $\tau$ on non-void open $\Omega \subseteq \mathbb C^n$ is not metrizable for any integer $n~\geqslant~1.$
\end{corollary}
\begin{proof}
Since every metrizable topological space is first countable, the corollary follows from Proposition \ref{p1.6}.
\end{proof}

M. Brelot in \cite{Br} considers a fine topology generated by a cone of lower-semicontinuous  functions of the form $f: X\to (-\infty, \infty]$. Every plurisuperharmonic function satisfies these conditions and such functions are just the negative of plurisubharmonic functions. Thus, the pluri-fine topology $\tau$ is an example of fine topologies studied in \cite{Br}.


\section{Separately continuous functions and the first Baire functions in the pluri-fine topology}

The following is a result from Mykhaylyuk's paper~\cite{My2} (see also \cite{My1}).

\begin{lemma}\label{t2.2}
Let $X$ be a topological space and let $X^m$ be a Cartesian product of $m\geqslant 2$ copies of $X$ with the usual product topology. Then for every $(m-1)$-Baire class function $g : X \to\mathbb R $ there is a separately continuous function $f : X^m \to\mathbb R$ such that $f(x, ..., x)=g(x)$ holds for every $x\in X$.
\end{lemma}

Let us denote by $t^{m}$ the Thychonoff topology (= product topology) on the product of $m$ copies of the topological space $(\Omega, \tau)$. The topology $t^{m}$ is the coarsest  topology on $\Omega^m$ making all projections $\pi_j:\Omega^m\to \Omega$, $j=1, \ldots, m$ continuous. Lemma~\ref{t2.2} directly implies the following.

\begin{lemma}\label{c2.3}
Let $m \geqslant 2$ be an integer number. For every $(m-1)$-Baire class function $g : \Omega \to\mathbb R$ in the pluri-fine topology $\tau$ there is a separately continuous function $f : \Omega^m\to \mathbb R$ in the Thychonoff topology $t^m$ such that $g$ is the diagonal of $f$.
\end{lemma}

The following theorem gives a ``pluri-fine'' analog of the implication $(i) \Rightarrow (ii)$ from Theorem \ref{t1.1}.

\begin{theorem}\label{t2.1}
Let $\Omega$ be a non-void open subset of\/ $\mathbb C^n$, $n\geqslant 1$ and let $m\geqslant 2$ be an integer number. For every $(m-1)$ Baire class function $g: \Omega\to\mathbb R$, in the pluri-fine topology $\tau$, there is a separately continuous function $f: \Omega^m\to\mathbb R$, in the pluri-fine topology $\tau_{m}$, such that
\begin{equation}\label{eq2.1}
g=f \circ d_m,
\end{equation}
where $d_m$ is the corresponding diagonal mapping.
\end{theorem}

\begin{proof}
By Lemma~\ref{c2.3}, it is sufficient to show that $t^{m}$ is weaker than $\tau_{m}$. From the definition of Thychonoff topology it follows at ones that $t^{m}$ is weaker than $\tau_{m}$ if and only if all projections $\pi_j\colon \Omega ^m \to \Omega$, $j\in \{1,...,m\}$, are continuous mappings on $(\Omega^m, \tau_{m})$. The continuity of these projections was proved in Lemma~\ref{l2.4}.
\end{proof}

\begin{proposition}\label{p2.6}
The equality
\begin{equation}\label{eq2.4}
H_{1}^{*}(X, \mathbb R)=B_{1}(X, \mathbb R)
\end{equation}
holds for every topological space $X$.
\end{proposition}
\begin{proof}
Let $X$ be a topological space and let $Y$ be an arcwise connected, locally arcwise connected, metrizable space. Then every $f\in H_{1}^{*}(X, Y)$, with separable $f(X)$, belongs to $B_{1}(X, Y )$ (see \cite{KaMy}). Hence $H_{1}^{*}(X, \mathbb R)\subseteq B_{1}(X, \mathbb R)$ holds.

It still remains to make sure that
\begin{equation}\label{eq2.5}
H_{1}^{*}(X, \mathbb R)\supseteq B_{1}(X, \mathbb R)
\end{equation}
is valid for every topological space $X$. The following is a simple modification of well known arguments.

Let $f \in B_{1}(X, \mathbb R)$. Consider a sequence $(f_n)_{n\in \mathbb N}$ of continuous real valued functions on $X$ such that the limit relation $f(x) = \mathop{\lim}\limits_{n\to\infty}f_{n}(x)$ holds for every $x\in X$. Let $(\varepsilon_{m})_{m\in\mathbb N}$ be a strictly decreasing sequence of positive real numbers with
\begin{equation}\label{eq2.6}
\lim_{m\to\infty}\varepsilon_{m}=0.
\end{equation}
Let us prove the equality
\begin{equation}\label{eq2.7}
f^{-1}(-\infty, a)=\bigcup_{m=1}^{\infty}\bigcup_{p=1}^{\infty}\left(\bigcap_{k\geqslant p}^{\infty}f^{-1}_{k}(-\infty, a-\varepsilon_{m}])\right)
\end{equation}
for every $a\in \mathbb R.$ It is sufficient to show that for every $x\in f^{-1}(-\infty, a)$ there are $m$, $p\in\mathbb N$ such that
\begin{equation}\label{eq2.8}
x\in\bigcap_{k\geqslant p}^{\infty}f^{-1}_{k}(-\infty, a-\varepsilon_{m}].
\end{equation}
Let $x\in f^{-1}(-\infty, a)$. Then we have $\mathop{\lim}\limits_{n\to\infty}f_{n}(x)<a$. The last inequality and \eqref{eq2.7} imply   $\mathop{\lim}\limits_{n\to\infty}f_{n}(x)<a-\varepsilon_{m_1}$ for some $m_1$. Consequently, there is $p\in\mathbb N$ such that $f_{n}(x)<a-\varepsilon_{m_1}$ for all $n\geqslant p$, that is
\begin{equation*}
x\in\bigcap_{k\geqslant p}^{\infty}f^{-1}_{k}(-\infty, a-\varepsilon_{m}).
\end{equation*}
Since the sequence $(\varepsilon_{m})_{m\in\mathbb N}$ is strictly decreasing, the inclusion $$(-\infty, a-\varepsilon_{m})\subseteq (-\infty, a-\varepsilon_{m+1}]$$ follows for every $m$. Hence \eqref{eq2.8} holds with $m = m_{1} + 1$.

Note now that $f^{-1}_{k}(-\infty, a-\varepsilon_{m}]$ is functionally closed as a zero-set of the continuous function
\begin{equation*}
g_{k, m, a}(x):=\min(\max(f(x)-f(a-\varepsilon_{m}); 0); 1).
\end{equation*}
Since each countable intersection of functionally closed sets is functionally closed \cite[p. 42--43]{En}, equality \eqref{eq2.7} implies $f^{-1}(-\infty, a)\in F_{\sigma}^{*}$. Moreover, if
$g = -f$ and $b = -a$, then
$f^{-1}(a, \infty)=g^{-1}(-\infty, b)$ holds.
Hence, the set $f^{-1}(a, \infty)$ belongs to $F_{\sigma}^{*}$.

We can now easily prove \eqref{eq2.5}. Indeed, it is sufficient to show that $\{x: a<f(x)<b\}$ is a countable union of functionally closed sets for every $f\in B_{1}(X, \mathbb R)$ and every interval $(a, b)\subseteq \mathbb R$. Using \eqref{eq2.7}, we obtain
\begin{equation}\label{eq2.9}
f^{-1}(a, b)=\left(\bigcup_{i=1}^{\infty}H_{i}\right)\cap\left(\bigcup_{i=1}^{\infty}F_{i}\right)=\bigcup_{i, j=1}^{\infty}(H_{i}\cap F_{j}),
\end{equation}
where all $H_i$ and $F_j$ are functionally closed. It was mentioned above that the countable intersection of functionally closed sets is functionally closed. Hence, by \eqref{eq2.9}, $f^{-1}(a, b)\in F_{\sigma}^{*}$, so that \eqref{eq2.5} follows.
\end{proof}

\begin{corollary}\label{c2.5*}
The equality $B_1(\Omega,\mathbb R)=H^*_1(\Omega, \mathbb R)$ holds if a non-void open set $\Omega \subseteq \mathbb C^n$ is endowed by the pluri-fine topology $\tau$.
\end{corollary}
This corollary and Theorem~\ref{t2.2} imply the following result.

\begin{theorem}\label{t2.5}
Let $\Omega$ be a non-void open subset of $\mathbb C^n$, $n\geqslant 1$ and let $g : \Omega\to\mathbb R$ be a first functional Lebesgue class function on $(\Omega, \tau)$. Then there is a separately continuous function $f : \Omega^2 \to\mathbb R$ in the pluri-fine topology $\tau_{2}$ on $\Omega^2$ such that $g$ is the diagonal of $f$.
\end{theorem}

The proof of the next proposition  is a variantion of Luke\v{s}--Zaji\v{c}ek's method from \cite{LZ1, LZ2}.

\begin{proposition}\label{p2.7}
Let $X$ be a topological space. Then, for every $f : X \to \mathbb R$, the following conditions are equivalent.
\begin{enumerate}
\item The function $f$ belongs to $B_{1}(X, \mathbb R)$.
\item For each couple of real numbers $a, b$ with $a< b$ there are $H_1, H_2\in F_{\sigma}^{*}$ such that
\begin{equation}\label{eq2.10}
f^{-1}(a, +\infty)\supseteq H_{1}\supseteq f^{-1}(b, +\infty),
\end{equation}
\begin{equation}\label{eq2.11}
f^{-1}(-\infty, b)\supseteq H_{2}\supseteq f^{-1}(-\infty, a).
\end{equation}
\end{enumerate}
 \end{proposition}

\begin{proof}
The implication $\textrm{(i)}\Rightarrow\textrm{(ii)}$ follows from \eqref{eq2.4} and \eqref{eq1.2}. Let us prove $\textrm{(ii)}\Rightarrow\textrm{(i)}$. Using \eqref{eq2.5} and \eqref{eq2.9}, it is easy to see that we need only to make sure the statements
\begin{equation*}
f^{-1}(a, \infty)\in F_{\sigma}^{*}\quad\text{and}\quad f^{-1}(-\infty, a)\in F_{\sigma}^{*}
\end{equation*}
for every $a\in\mathbb R.$ Suppose \eqref{eq2.10} holds. Then, for every $m\in\mathbb N$, there is $H^{m}\in F_{\sigma}^{*}$ such that
\begin{equation*}
f^{-1}\left(a+\frac{1}{m}, +\infty\right)\subseteq H^{m}\subseteq f^{-1}\left(a+\frac{1}{m+1}, +\infty\right).
\end{equation*}
Consequently,
\begin{equation*}
f^{-1}(a, \infty)=\bigcup_{m=1}^{\infty}f^{-1}\left(a+\frac{1}{m}, +\infty\right)\subseteq\bigcup_{m=1}^{\infty}H^{m}
\end{equation*}
\begin{equation*}
\subseteq \bigcup_{m=1}^{\infty}f^{-1}\left(a+\frac{1}{m+1}, +\infty\right)=f^{-1}(a, \infty).
\end{equation*}
Thus,
\begin{equation*}
f^{-1}(a, \infty)=\bigcup_{m=1}^{\infty}H^{m}.
\end{equation*}
It implies  $f^{-1}(a, \infty)\in F_{\sigma}^{*},$ because every countable union of sets from $F_{\sigma}^{*}$ belongs to $F_{\sigma}^{*}.$

Similarly, using \eqref{eq2.11}, we can prove that $f^{-1}(-\infty, a)\in F_{\sigma}^{*}.$ The implication $\textrm{(ii)}\Rightarrow\textrm{(i)}$ follows.
\end{proof}
In the following corollary we consider the classes $B_1(\Omega,\mathbb R)$ and $F^*_{\sigma}$ with respect to the pluri-fine topology $\tau$ on $\Omega$.

\begin{corollary}\label{c2.8}
Let $\Omega$ be a non-void open subset of $\mathbb C^n$, $n\geqslant 1$, and let $f$ be a real valued function on $\Omega$. Then $f$ belongs to $B_1(\Omega,\mathbb R)$ if and only if, for each couple $a$, $b \in \mathbb R$, with $a<b$, double inclusions \eqref{eq2.10} and \eqref{eq2.11} hold for some $H_1$, $H_2 \in F^*_{\sigma}$.
\end{corollary}

\section*{Acknowledgment} The research of the first author was supported by a grant received from TUBITAK within 2221-Fellowship Programme for Visi\-ting Scientists and Scientists on Sabbatical Leave and as a part of EUMLS project with grant agreement PIRSES-GA-2011-295164.


\section*{References}
\begin{thebibliography}{99}

\bibitem{Ba} {\it R. Baire,} Sur les fonctions de variables r\'{e}eles //
Annali di Mat., {\bf 3} (1899), 1--123.

\bibitem{BT} {\it E. Bedford, B. Taylor,} Fine topology, \u{S}ilov boundary and $(dd^{c})^{n}$ //
 J. Funct. Anal., {\bf 72} (1987), 225--251.

\bibitem{Br} {\it M. Brelot,}  On Topologies and Boundaries in Potential Theory. Enlarged ed. of a course of lectures delivered in 1996. Lecture
Notes in Math., \textbf{175}, Springer Verlag, Berlin-Heidelberg-New
York, (1971).


\bibitem{En} {\it  R. Engelking,} General Topology. Rev. and compl. ed., Sigma Series in Pure Mathematics, \textbf{6}, Heldermann Verlag, Berlin, (1989).

\bibitem{Fug} {\it  B. Fuglede,} Fonctions harmoniques et fonctions finement harmoniques //
Ann. Inst. Fourier , {\bf 24} (4) (1974), 77--91.

\bibitem{Ha} {\it  H. Hahn,} Theorie der Reellen Funktionen, Erster Band, Springer, Berlin, (1921).


\bibitem{Ho} {\it  L. H\"{o}rmander,} Notions of Convexity, Progress in Mathematics, \textbf{127},
Birkha\"{u}ser, Boston-Basel-Berlin, (1994).

\bibitem{KaMy} {\it  O. Karlova, V. Mykhajlyuk,} Functions of the first Baire class with
values in metrizable spaces //
Ukr. Math. J., {\bf 58} (4) (2006), 640--644.

\bibitem{Le1}{\it H. Lebesgue,} Sur l'approximation des fonctions //
Darboux Bull., \textbf{22} (1898),  278--287.

\bibitem{Le2} {\it  H. Lebesgue,} Sur les fonctions repr\'{e}sentables analytiquement //
Journ. de Math., \textbf{6} (1) (1905),  139--216.

\bibitem{LMZ} {\it  J. Luke\v{s}, J. Mal\'{y}, L. Zaji\v{c}ek,} Fine Topology Methods in Real Analysis and Potential Theory, Lectures Notes in Mathmatics, \textbf{1189}, Springer Verlag, Berlin-Heidelberg-New York-London-Paris-Tokio, (1986).

\bibitem{LZ1} {\it  J. Luke\v{s}, L. Zaji\v{c}ek,} The intersection of $G_{\delta}$ sets and fine topologies//
Commentat. Math. Univ. Carol., {\bf 18} (1977), 101--104.

\bibitem{LZ2} {\it  J. Luke\v{s}, L. Zaji\v{c}ek,} When finely continuous functions are of the first class of
Baire // Commentat. Math. Univ. Carol., {\bf 18} (1977), 647--657.

\bibitem{MW1} {\it  S. El Marzguioui, J. Wiegerinck,} The pluri-fine topology is locally connected // Potential Anal., {\bf 25}(3) (2006), 283--288.

\bibitem{MW2} {\it  S. El Marzguioui, J. Wiegerinck,} Connectedness in the pluri-fine topologi, Functional Analysis and Complex Analysis, Istanbul  17--21 (2007), Contemp. Math. \textbf{481}, 105--115;  Amer. Math. Soc., Providence, RI (2009).

\bibitem{My1} {\it  V. Mykhajlyuk,} Construction of separately continuous functions with given restrictions // Ukr. Math. J., {\bf 55} (5) (2003), 866--872.

\bibitem{My2} {\it  V. Mikha\u\i lyuk,} Construction of separately continuous functions of $n$ variables with given restrictions // Ukr. Mat. Visn., {\bf 3} (3) (2006), 374--381.
\end{thebibliography}
\end{document}